\documentclass[12pt,reqno]{amsart}
\usepackage{amsmath,amssymb, amsthm}

\usepackage[colorlinks,urlcolor=red]{hyperref}
\usepackage{graphicx}\usepackage{multirow}
\usepackage{cite}
\usepackage{algpseudocode}
\usepackage{algorithm}
\theoremstyle{plain}
\newtheorem{thm}{Theorem}[section]
\newtheorem{cor}[thm]{Corollary}
\newtheorem{lem}[thm]{Lemma}
\newtheorem{alg}[thm]{Algorithm}
\newtheorem{prop}[thm]{Proposition}
\numberwithin{equation}{section} \theoremstyle{definition}
 \newtheorem{dfn}[thm]{Definition}

 \newtheorem{rem}[thm]{Remark}

\allowdisplaybreaks

\numberwithin{equation}{section}
\renewcommand{\to}{\longrightarrow}
\begin{document}
\title[Splitting algorithms of common solutions to equilibrium and inclusion]
{Splitting Algorithms of Common Solutions Between Equilibrium and Inclusion Problems  on Hadamard Manifolds}
\author[K. Khammahawong, P. Kumam and P. Chaipunya]
{Konrawut Khammahawong$^{1,2}$, Poom Kumam$^{1, 2,3, \S}$ and Parin Chaipunya$^{1,2,3}$
}
\date{}
\thanks{\it $^\S$ Corresponding author: poom.kumam@mail.kmutt.ac.th (P. Kumam)}
\maketitle

\begin{center}
{\footnotesize $^{1}$ Department of Mathematics, Faculty of Science, King Mongkut's University of Technology Thonburi (KMUTT), 126 Pracha Uthit Rd., Bang Mod,Thung Khru, Bangkok 10140, Thailand
	 \vskip 2mm
	 $^{2}$KMUTTFixed Point Research Laboratory,  KMUTT-Fixed Point Theory and Applications Research Group, Department of Mathematics, Faculty of Science, King Mongkut’s University of Technology Thonburi (KMUTT), 126 Pracha-Uthit Road, Bang Mod, Thrung Khru, Bangkok 10140, Thailand
 \vskip 2mm
$^{3}$Center of Excellence in Theoretical and Computational Science (TaCS-CoE), SCL 802 Fixed Point Laboratory, Science Laboratory Building, King Mongkut’s University of Technology Thonburi (KMUTT), 126 Pracha-Uthit Road, Bang Mod, Thrung Khru, Bangkok 10140, Thailand   \vskip 2mm
Email addresses: {k.konrawut@gmail.com  (K. Khammahawong),  \\
	poom.kumam@mail.kmutt.ac.th (P. Kumam)} and parin.cha@mail.kmutt.ac.th (P. Chaipunya) }

\end{center}

%---------------------------------------------------------------------------------
\maketitle

\hrulefill

\begin{abstract}
The aim of this article is to introduce an iterative algorithm for finding a common solution from the set of an equilibrium point for a bifunction and  the set of a singularity  of an inclusion problem on an Hadamard manifold. We also discuss some particular cases of the problem by the proposed algorithm. The convergence  of a sequence generated by the proposed algorithm is proved under mild assumptions. Moreover, we apply our results to solving minimization problems and minimax problems.
\end{abstract}

\begin{footnotesize}
\end{footnotesize}
% ----------------------------------------------------------------
\maketitle

\hrulefill

\section{Introduction}
Equilibrium problem (EP) was firstly introduced by Fan \cite{MR0341029} and extensively developed later  by Blum and Oettli\cite{MR1292380}.  Let $H$ be a real Hilbert space, $K$ a nonempty closed convex subset of $H$ and $F : K \times K \to \mathbb{R}$ a bifunction satisfying $F(x,x) =0$, for all $x \in K$. The equilibrium problem for a bifunction $F$ is to find $x^* \in K$ such that
\begin{equation}\label{eq EP1}
F(x^*,y) \geq 0, \quad \forall y \in K. 
\end{equation}
Herein,  $F$ is said to be the {\it equilibrium bifunction}. The theory of equilibrium problem plays a vital role in nonlinear problems, e.g., variational inequalities, optimization problems, Nash equilibrium problems,  complementarity problems  and so on, (see, for example \cite{MR2503647,MR1171603,MR3702600,MR3433705} and the references therein). 

In 1976, Rockafellar \cite{MR0410483} considered the following inclusion problem:
\begin{equation}\label{eq inclusion hilbert}
{\rm find} \ x \in K \ {\rm such \ that} \ 0 \in A(x),
\end{equation}
where $A: K \to 2^H$ is a given maximal monotone operator. The classical method for solving inclusion problem \eqref{eq inclusion hilbert} is the proximal point method. The method was firstly introduced by Martinet \cite{MR0298899} for convex minimization and further generalized by Rockafellar \cite{MR0410483}. Many problems in nonlinear analysis, optimization problem, convex programming problem, variational inequality problem, PDEs, economics are reduced to finding a singularity of the problem \eqref{eq inclusion hilbert}, see for example \cite{MR2815119,MR1776022,MR1970677,MR1761939,MR1788273} and the references therein.

 During the last decade, many issues in nonlinear analysis such as fixed point theory, convex analysis, variational inequality, equilibrium theory, and optimization theory have,	 been magnified  from linear setting, namely, Banach spaces or Hilbert spaces, etc., to nonlinear system because the problems cannot be posted in the linear space and require a  manifold structure (not necessary with linear structure).  The main advantages of these extensions are that non-convex problems in the general sense  are  transformed into convex problems, and constrained problems also transform into unconstrained  problems. Eigenvalue optimization problems \cite{MR1297990} and geometric models for the human spine \cite{MR1918655} are typical examples of the situation. Therefore, many authors have focused on extension and development of nonlinear problems techniques on the Riemannian manifold, see for examples  \cite{MR2824431, MR1928039, MR2869729,MR3131827} and the reference therein. 
 
  In 2012, Calao et al. \cite{MR2869729} studied the equilibrium problems on a Hadamard manifold.  Let $M$ be an Hadamard manifold, $TM$ is the tangent bundle of $M$, $K$ a nonempty closed geodesic convex subset of $M$, and $F : K \times K \to \mathbb{R}$ a bifunction satisfies $F(x,x) = 0$, for all $x \in K$. Then the equilibrium problem on the Hadamard manifold is to find $x^* \in K$ such that 
 \begin{equation}\label{eq EP}
 F(x^*,y) \geq 0, \quad \forall y \in K.
 \end{equation}
 We denote $EP(F)$ by the set of a equilibrium point of the equilibrium problem \eqref{eq EP}. They studied the existence of an equilibrium point for a bifunction under suitable conditions and applied their results to solving mixed variational inequality problems, fixed point problems and Nash equilibrium problems in Hadamard manifolds. The authors also introduced Picard iterative method to approximate a solution of the problem \eqref{eq EP}. However, Wang et al. \cite{MR3912856} found some gaps in  the existence proof of the mixed variational inequalities and the domain of the resolvent for the equilibrium problems in \cite{MR2869729}. 
 
 The inclusion problem \eqref{eq inclusion hilbert} is considered by Li et al. \cite{MR2506692}  in Hadamard manifolds, and it reads as follows:
 \begin{equation}\label{eq inclusion 2}
 {\rm  find } \ x \in K \ {\rm such \ that} \  {\bf 0} \in A(x),
 \end{equation}
 where $A: K \to 2^{TM}$ be a multivalued vector field on a Hadamard manifold and ${\bf 0}$ denotes the zero section of $TM$. We denote $A^{-1}(\textbf{0})$ by the set of a singularity of the inclusion problem \eqref{eq inclusion 2}. The authors also extended the general proximal point method from Euclidean spaces to Hadamard manifolds for solving the inclusion problem \eqref{eq inclusion 2}. 
 
 Motivated by above results, we introduce iterative algorithm for finding a common solution of the equilibrium problem \eqref{eq EP} and the inclusion problem \eqref{eq inclusion 2} on Hadamard manifolds. Our proposed algorithm can be regraded as the double-backward method for the two underlying problems.

 The rest of this paper is organized in the following: In Section \ref{sec2}, we give some basic concepts and fundamental results of Riemannian manifolds as well as some useful results. In Section \ref{sec3}, we introduce the problem of finding $x \in EP(F) \cap A^{-1}({\bf 0})$, which is  a common solution of the set of an equilibrium point of a bifunction and the set of  a singularlity of the multivalued vector field.  We propose an iterative algorithm  and establish convergence results of a sequence generated by the proposed algorithm converges to a common solution of the proposed problem on Hadamard manifolds.   In the last section, we devote our results to minimization problems and minimax problems on Hadamard manifolds.

 \section{Preliminaries}\label{sec2}
In this section, we recall some fundamental definitions, properties, useful results, and notations of Riemannian geometry. Readers refer to some textbooks \cite{MR1390760, MR1138207, MR1326607} for more details. 

Let $M$ be a connected finite-dimensional manifold.  For $p \in M$, we denote $T_pM$  the {\it tangent space} of $M$ at $p$ which is a vector space of the same dimension as $M$, and by $TM = \bigcup_{p \in M}T_pM$ the  {\it tangent bundle} of $M$. We always suppose that $M$ can be endowed with a Riemannian metric $\langle \cdot , \cdot \rangle_p$, with corresponding norm denoted by $\|\cdot\|_p$, to become a {\it Riemannian manifold}. The angle $\angle_p(u,v)$ between $u,v \in T_pM \ (u,v \neq 0)$ is set by $\cos \angle_p(u,v) = \dfrac{\langle u , v \rangle_p}{\|u\|\|v\|}.$ If there is no confusion, we denote $\langle \cdot , \cdot \rangle := \langle \cdot , \cdot \rangle_p$, $\|\cdot\| := \|\cdot\|_p$ and $\angle(u,v) := \angle_p(u,v)$.  Let  $\gamma : [a,b] \to M$ be a piecewise smooth curve  joining $\gamma (a) = p$ to $\gamma (b) = q$, we define the length of the curve $\gamma$ by using the metric as 
$$L(\gamma) = \int_a^b \|\gamma^{'}(t)\|dt,$$
minimizing the length function over the set of all such curves, we obtain a Riemannian distance $d(p,q)$ which induces the original topology on $M$.

Let $\nabla$ be a Levi-Civita connection associated to $(M, \langle \cdot , \cdot \rangle)$. Given $\gamma$ a smooth curve, a smooth vector field $X$ along $\gamma$ is said to be {\it parallel} if  $\nabla_{\gamma^{'}}X ={\bf 0}$. If $\gamma^{'}$ itself is parallel, we say that $\gamma$ is a {\it geodesic}, and in this case $\|\gamma^{'}\|$ is a constant. When $\|\gamma^{'}\| = 1$, then $\gamma$ is said to be {\it normalized}. A geodesic joining $p$ to $q$ in $M$ is said to be a {\it minimal geodesic} if its length equals to $d(p,q)$.

A Riemannian manifold is complete if for any $p \in M$ all geodesic emanating from $p$ are defined for all $t \in \mathbb{R}$. From the Hopf-Rinow theorem we know that if $M$ is complete then any pair of points in $M$  can be joined by a minimal geodesic. Moreover, $(M,d)$ is a complete metric space and every bounded closed subsets are compact.

Let $M$ be a complete Riemannian manifold and $ p \in M$. The exponential map $\exp_p : T_pM \to M$ is defined as $\exp_pv = \gamma_v(1,x)$, where $\gamma(\cdot) = \gamma_v(\cdot,x)$ is the geodesic starting at $p$ with velocity $v$ (i.e., $\gamma_v(0,p) = p$ and $\gamma^{'}_v(0,p) = v$). Then, for any value of $t$, we have $\exp_ptv = \gamma_v(t,p)$ and $\exp_p{\bf 0} = \gamma_v(0,p) =p$. Note that the exponential $\exp_p$ is differentiable on $T_pM$ for all $p \in M$. It well known that the derivative $D \exp_p({\bf 0})$ of $\exp_p({\bf 0})$ is equal to the identity vector of $T_pM.$ Therefore, by the inverse mapping theorem, there exists an inverse exponential map $\exp^{-1} : M \to T_pM$. Moreover, for any $p,q \in M$, we have $d(p,q) = \|\exp_p^{-1}q\|$.

A complete simply connected Riemannian manifold of non-positive sectional  curvature is said to be an {\it Hadamard manifold.} Throughout the remainder of the paper, we always assume that $M$ is a finite-dimensional Hadamard manifold.  The following proposition is well-known and will be useful.

\begin{prop}{\rm \cite{MR1390760}} \label{prop1}
	Let $p \in M$. The $\exp_p : T_pM \to M$ is a diffeomorphism, and for any two points $p,q \in M$ there exists a unique normalized geodesic joining $p$ to $q$, which is can be expressed by the formula
	$$\gamma (t) = \exp_p t \exp_p^{-1}q, \quad \forall t \in [0,1].$$
\end{prop}
\noindent This proposition  yields  that $M $ is diffeomorphic to the Euclidean space $\mathbb{R}^n$. Then, $M$ has same topology and differential structure as $\mathbb{R}^n$. Moreover, Hadamard manifolds and Euclidean spaces have some similar geometrical properties. One of the most important proprieties is  illustrated in the following propositions.

A geodesic triangle $\bigtriangleup(p_1,p_2,p_3)$ of a Riemannian manifold $M$ is a set consisting of three points $p_1$, $p_2$ and $p_3$, and three minimal geodesic $\gamma_i$ joining $p_{i}$ to $p_{i+1}$ where $i =1,2,3 \ ({\rm mod} 3)$.

\begin{prop}{\rm \cite{MR1390760}} \label{prop2}
	Let $\bigtriangleup(p_1,p_2,p_3)$ be a geodesic triangle in $M$. For each $i = 1, 2, 3  \ ({\rm mod} 3)$, given $\gamma_{i} : [0, l_{i}] \to M$ the geodesic joining $p_{i}$ to $p_{i+1}$ and set $l_{i} := L(\gamma_{i})$, $\alpha_{i} : \angle(\gamma_{i}^{'}(0), -\gamma_{i-1}^{'}(l_{i-1}))$. Then
	\begin{equation}\label{eq tri1}
	\alpha_1 + \alpha_2 + \alpha_3 \leq \pi;
	\end{equation}
	\begin{equation}\label{eq tri2}
	l_{i}^2 + l_{i+1}^2 - 2l_{i}l_{i+1}\cos \alpha_{i+1} \leq l_{i-1}^2.
	\end{equation}
	In the terms of the distance and the exponential map, the inequality \eqref{eq tri2} can be rewritten as
	\begin{equation}\label{eq tri3}
	d^2(p_{i},p_{i+1}) + d^2(p_{i+1},p_{i+2}) - 2\langle \exp^{-1}_{p_{i+1}}p_{i} , \exp^{-1}_{p_{i+1}}p_{i+2} \rangle \leq d^2(p_{i-1},p_{i}),
	\end{equation}
	where $\langle \exp^{-1}_{p_{i+1}}p_{i} , \exp^{-1}_{p_{i+1}}p_{i+2} \rangle = d(p_{i},p_{i+1})d(p_{i+1},p_{i+2})\cos \alpha_{i +1}.$
\end{prop}
The following  relation between geodesic triangles in Riemannian manifolds and triangles in $\mathbb{R}^2$ can be referred to \cite{MR1744486}.

\begin{lem}\label{lem tri}{\rm  \cite{MR1744486}}
	Let $\bigtriangleup(p_1,p_2,p_3)$ be a geodesic triangle in $M$. Then there exists a triangle $\bigtriangleup(\overline{p_1}, \overline{p_2}, \overline{p_3})$ for  $\bigtriangleup(p_1,p_2,p_3)$ such that $d(p_{i},p_{i+1}) = \|\overline{p_i} - \overline{p_{i+1}}\|$, indices taken modulo $3$; it is unique up to an isometry of $\mathbb{R}^2$.
\end{lem}

The triangle $\bigtriangleup(\overline{p_1}, \overline{p_2}, \overline{p_3})$  in Lemma \ref{lem tri} is said to be a  {\it comparison triangle} for $\bigtriangleup(p_1,p_2,p_3)$. The  geodesic side from $x$ to $y$ will be denoted $[x,y]$. A point $\overline{x} \in [\overline{p_1}, \overline{p_2}]$ is said to be a {\it comparison point} for $x \in [p_1,p_2]$ if $\|\overline{x} - \overline{p_1}\| = d(x,p_1)$. The interior angle of  $\bigtriangleup(\overline{p_1}, \overline{p_2}, \overline{p_3})$ at $\overline{p_1}$ is said to be  the {\it comparison angle} between $\overline{p_2}$ and $\overline{p_3}$ at $\overline{p_1}$ and is denoted $\angle_{\overline{p_1}}(\overline{p_2},\overline{p_3})$. With all notation as in the statement of Proposition \ref{prop2}, according to the law of cosine, \eqref{eq tri2} is valid if and only if
\begin{equation}\label{eq angle}
\langle \overline{p_2} - \overline{p_1}, \overline{p_3} - \overline{p_1} \rangle_{\mathbb{R}^2} \leq \langle \exp^{-1}_{p_{1}}p_2, \exp^{-1}_{p_{1}}p_{3} \rangle
\end{equation}
or, 
$$\alpha_1 \leq \angle_{\overline{p_1}}(\overline{p_2},\overline{p_3})$$
or, equivalent, $\bigtriangleup(p_1,p_2,p_3)$ satisfies the CAT($0$) inequality and that is, given a comparison triangle $\overline{\bigtriangleup} \subset \mathbb{R}^2$ for $\bigtriangleup(p_1,p_2,p_3)$ for all $x,y \in \bigtriangleup$,
\begin{equation}\label{eq less}
d(x,y) \leq \|\overline{x} - \overline{y}\|,
\end{equation}
where $\overline{x}, \overline{y} \in \overline{\bigtriangleup}$ are the respective comparison points of $x,y$.

A subset $K$ is called {\it geodesic convex} if for every two points $p$ and $q$ in $K$, the geodesic joining $p$ to $q$ is contained in $K$, that is, if $\gamma : [a,b] \to M$ is a geodesic such that $p = \gamma(a)$ and $q = \gamma(b)$, then $\gamma((1-t)a + tb) \in K$ for all $t \in [0,1].$ 

A real function $f: M \to \mathbb{R}$ is called {\it geodesic convex} if for any geodesic $\gamma$ in $M$, the composition function $f \circ \gamma : [a,b] \to \mathbb{R}$ is convex, that is,
$$(f\circ\gamma)(ta + (1-t)b) \leq t(f \circ \gamma)(a) + (1-t)(f \circ \gamma)(b)$$
where $a,b \in \mathbb{R}$, and $ 0 \leq t \leq 1 $.  

\begin{prop}\label{prop dis geodesic} {\rm \cite{MR1390760}} 
	Let $d : M \times M \to \mathbb{R}$ be the distance function. Then $d(\cdot,\cdot)$ is a geodesic convex function with respect to the product Riemannian metric, that is, for any pair of geodesics $\gamma_1 : [0,1] \to M$ and $\gamma_2 : [0,1] \to M$ the following inequality holds for all $t \in [0,1]$
	$$d(\gamma_1(t), \gamma_2(t)) \leq (1-t)d(\gamma_1(0),\gamma_2(0)) + td(\gamma_1(1),\gamma_2(1)).$$
\end{prop}
In particular, for each $y \in M$, the function $d(\cdot,y) : M \to \mathbb{R}$ is a geodesic convex function.
\begin{dfn}\label{def fejer}{\rm \cite{MR1928039}}
	Let $K$ be a nonempty subset of $M$ and $\{x_n\}$ be a sequence in $M$. Then $\{x_n\}$ is said to be {\it Fej\'{e}r convergent} with respect to $K$ if for all $p \in K$ and $n \in \mathbb{N},$
	$$d(x_{n+1},p) \leq d(x_n,p).$$
\end{dfn}
\begin{lem}\label{lem fejer}{\rm \cite{MR1928039}}
	Let $K$ be a nonempty subset of $M$ and $\{x_n\} \subset X$ be a sequence in $M$ such that $\{x_n\}$ be a Fej\'{e}r convergent with respect to $K$. Then the following hold:
	\begin{itemize}
		\item[{\rm (i)}] For every $p \in K$, $d(x_n,p)$ converges;
		\item[{\rm (ii)}] $\{x_n\}$ is bounded;
		\item[{\rm (iii)}] Assume that every cluster point of $\{x_n\}$ belongs to $K$. \\ Then $\{x_n\}$ converges to a point in $K$.
	\end{itemize}
\end{lem}

Recall that for all $x,y \in \mathbb{R}^2$,
$$\|tx + (1-t)y\|^2 = t\|x\|^2 + (1-t)\|y\|^2 -t(1-t)\|x-y\|^2, \quad \forall t \in [0,1]$$

Given $K$ be a nonempty subset of $M$. Let   $\mathfrak{X}(K)$ denote to the set of all multivalued vector fields $A: K \to 2^{TM}$ such that $A(x) \subseteq T_xM$ for each $x \in K,$ and denote $D(A)$  the domain of $A$ defined by $D(A) = \{x \in K : A(x) \neq \emptyset\}.$

\begin{dfn}\label{def monotone setvalued} \cite{MR1787612}
	A vector field $A \in \mathfrak{X}(K)$  is said to be 
	\begin{itemize}
		\item[{\rm (i)}] {\it monotone} if for all $x,y \in D(A)$
		$$\langle u, \exp^{-1}_xy\rangle \leq \langle v, - \exp^{-1}_yx\rangle, \quad \forall u \in A(x) \ {\rm and} \ \forall v \in A(y);$$
		\item[{\rm (ii)}] {\it maximal monotone}  if it is monotone and for all $x \in K$ and $u \in T_xK$, the condition
		$$\langle u, \exp^{-1}_xy\rangle \leq \langle v, - \exp^{-1}_yx\rangle, \quad \forall y \in D(A) \ {\rm and} \ \forall v \in A(y),$$
		implies that $u \in A(x).$
	\end{itemize}
\end{dfn}

The concept of Kuratowski semicontinuity on Hadamard manifolds was introduced by Li et al. \cite{MR2506692}.
\begin{dfn}\label{def kuratowski} \cite{MR2506692}
	Let a vector field $A \in \mathfrak{X}(K)$  and $x_0 \in K$. Then $A$ is said to be {\it upper Kuratowski semicontinuous at $x_0$} if for any sequences $\{x_n \} \subseteq K$ and $\{v_n\} \subset TM$ with each $v_n \in A(x_n)$, the relations $\lim_{n \to \infty}x_n = x_0$ and $\lim_{n \to \infty}v_n = v_0$ imply that $v_0 \in A(x_0)$. Moreover, $A$ is said to be  {\it upper Kuratowski semicontinuous on $K$} if it is upper Kuratowski semicontinuous for each $x \in K$.
\end{dfn}

%\begin{lem}\label{lem maximal} {\rm \cite{MR2506692}}
%	If $A \in \mathfrak{X}(K)$ is a maximal monotone vector field, then it  is upper Kuratowski semicontinuous on $K$.
%\end{lem}

%Let $K$ be a nonempty closed geodesic convex subset of $M$. The {\it projection operator} $P_K(\cdot) : M \to K$ is defined for any $x \in M$ by
%$$P_K(x) := \{\overline{x} : d(x,\overline{x}) \leq d(x,y), \forall y \in K \}.$$
%The projection operator $P_K$ is firmly nonexpansive as described in the following proposition {\rm  \cite{MR2824431}}.
%
%\begin{prop}\label{prop projection} {\rm  \cite{MR2824431}}
%	Let $K$ be a nonempty closed geodesic convex subset of $M$. Then the following assertions holds:
%	\begin{itemize}
%		\item[(i)] $P_K$ is single-valued and firmly nonexpansive;
%		\item[(ii)] For all $x \in M, \ z = P_K(x)$ if and only if
%		$$\langle \exp_z^{-1}x, \exp^{-1}_zy\rangle \leq 0, \quad \forall y \in K.$$
%	\end{itemize}
%\end{prop}

The definition of the resolvent  of a multivalued vector field and firmly nonexpansive mappings on Hadamard manifolds was introduced by  Li et al. \cite{MR2824431}.

\begin{dfn}\label{def resolvent} {\rm \cite{MR2824431}}
	Let a vector field $A \in \mathfrak{X}(K)$ and $\lambda \in (0,\infty)$. The $\lambda$-{\it resolvent}  of $A$  is a multivalued map $J^A_{\lambda} : K \to 2^K$ defined by
	$$J_{\lambda}^A(x) := \{z \in K : x \in \exp_{z}\lambda A(z) \}, \quad \forall x \in K.$$
\end{dfn}

\begin{rem}{\rm \cite{MR2824431}}\label{rem resolvent}
	Let $\lambda > 0.$ By the definition of the resolvent of a vector field, then the range of the resolvent $J^A_\lambda$ is contained the domain of $A$ and ${\rm Fix}(J^A_\lambda) = A^{-1}({\bf 0}).$
\end{rem}

\begin{dfn}\label{def fimly} {\rm \cite{MR2824431}}
	Let $K$ be a nonempty subset of $M$ and $T : K  \to M$ be a mapping. Then $T$ is called {\it firmly nonexpansive} if for all $x,y \in K$, the function $\Phi : [0,1] \to [0, \infty]$ defined by 
	$$\Phi(t) := d(\exp_x t \exp^{-1}_xTx, \exp_y t \exp^{-1}_yTy), \quad \forall t \in [0,1],$$
	is nonincreasing.
\end{dfn}
A mapping $T : K \to K$ is called {\it nonexpansive} if 
$d(T(x),T(y)) \leq d(x,y),$
for all $x,y \in K$, where $d(x,y)$ is a Riemannian distance. By definition of firmly nonexpansive, it easy to see that any firmly nonexpansive mapping is nonexpansive mapping. 

\begin{thm}\label{thm resolvent is fimly} {\rm \cite{MR2824431}}
	Let a vector field $ A \in \mathfrak{X}(K)$. The following assertions hold for any $\lambda >0$
	\begin{itemize}
		\item[(i)] The vector field $A$ is monotone if and only if $J^A_\lambda$ is single-valued and firmly nonexpansive.
		\item[(ii)] If $D(A) = K$, the vector field $A$ is maximal monotone if and only if $J^A_\lambda$ is single-valued, firmly nonexpansive and the domain $D(J^A_\lambda) = K$.
	\end{itemize}
	
\end{thm}

\begin{prop}\label{prop firmly = 0}  {\rm  \cite{MR2824431}}
	Let $K$ be a nonempty subset of $M$ and $T : K \to M$ be a firmly nonexpansive mapping. Then
	$$\langle \exp^{-1}_{Ty}x, \exp_{Ty}^{-1}y\rangle \leq 0$$
	holds for any $x \in {\rm Fix}(T)$ and any $y \in K$.
\end{prop}	

\begin{lem}\label{lem convergent resolvent} {\rm \cite{Al_Homidan_2019}}
	Let $K$ be a nonempty closed subset of $M$ and a vector field $A \in \mathfrak{X}(K)$ be a maximal monotone.
	Let $\{\lambda_n\} \subset (0,\infty)$ be a real sequence with $\lim_{n \to \infty}\lambda_n = \lambda >0$ and a sequence $\{x_n\} \subset K$ with $\lim_{n \to \infty}x_n = x \in K$ such that $\lim_{n \to \infty}J^A_{\lambda_n}(x_n) = y$. Then, $y = J_\lambda^A(x).$
\end{lem}

Let $K$ be a nonempty closed geodesic convex set in $M$ and $F : K \times K \to \mathbb{R}$ be a bifunction. We suppose the following assumptions:
\begin{itemize}
	\item[{\rm (A1)}] for all $x \in K$, $f(x,x) \geq 0;$ 
	\item[{\rm (A2)}] $F$ is monotone, that is, for all $x,y \in K,$ $F(x,y) + F(y,x) \leq 0;$
	\item[{\rm (A3)}] For every $y \in K$, $x \mapsto F(x,y)$ is upper semicontinuous;
	\item[{\rm (A4)}] For every $x \in K$, $y \mapsto F(x,y)$ are geodesic convex and lower semicontinuous;
	\item[{\rm (A5)}] $x \mapsto F(x,x)$ is lower semicontinuous;
	\item[{\rm (A6)}] There exists a compact set $L \subseteq M$ such that
	$$x \in K \setminus L \Longrightarrow [\exists y \in K \cap L \ {\rm such \ that} \ F(x,y) < 0 ].$$
\end{itemize}
Calao et al. \cite{MR2869729}  introduced the concept of resolvent of a bifunction on Hadamard manifold as follows: let $F : K \times K \to \mathbb{R}$, the resolvent of a bifunction $F$ is  a multivalued operator $T^F_\lambda : M \to 2^K$ such that for all $x \in M$
$$T^F_{r}(x) = \{z \in K : F(z,y) -\frac{1}{r} \langle \exp^{-1}_{z}x, \exp^{-1}_zy \rangle \geq 0, \ \forall y \in K  \}. $$ 

\begin{thm} \label{thm resolvent bifunction} {\rm \cite{MR2869729}}
	Let  $F : K \times K \to \mathbb{R}$ be a bifunction satisfying the following conditions:
	\begin{itemize}
		\item[{\rm (1)}] $F$ is monotone;
		\item[{\rm (2)}] for all $r >0, $ $T^F_r$ is properly defined, that is, the domain $D(T^F_r) \neq \emptyset$.
	\end{itemize}   Then for any $r >0,$
	\begin{itemize} 
		\item[{\rm (i)}] the resolvent $T^F_r$ is single-valued;
		\item[{\rm (ii)}] the resolvent $T^F_r$ is firmly nonexpansive;
		\item[{\rm (iii)}] the fixed point set of $T^F_r$ is the  equilibrium point set of $F$, 
		$${\rm Fix}(T^F_r) = EP(F);$$
		\item[{\rm (iv)}] If $D(T^F_r)$ is closed and geodesic convex,  the equilibrium point set $EP(F)$ is closed and geodesic convex.
	\end{itemize}
\end{thm}

\begin{thm} \label{thm domain resolvent bifunction} {\rm \cite{MR3912856}}
	Let  $F : K \times K \to \mathbb{R}$ be a bifunction satisfying the assumptions (A1)--(A3). Then $D(T^F_r) = M$.
\end{thm}

\begin{lem} \label{lem resolvent bifunction} {\rm \cite{MR3912856}}
	Let  $F : K \times K \to \mathbb{R}$ be a bifunction satisfying the assumptions (A1), (A3), (A4), (A5)  and (A6). Then there exists $z \in K$ such that
	$$F(z,y) -\frac{1}{r} \langle \exp^{-1}_{z}x, \exp^{-1}_zy \rangle \geq 0, \ \forall y \in K,$$
	for all $r >0$ and $x \in M$.
\end{lem}
%\begin{alg}\label{alg prox}
%Letting $i \in \mathbb{N}$ and having $x_i$, choose $x_{i+1}$ such that
%$$x_{i+1} = T^f_{r_i}(x_i).$$
%\end{alg}

\section{Main Results}\label{sec3}
In this paper, $K$ always denotes  a nonempty closed geodesic convex subset of an Hadmard manifold $M$, unless explicitly stated otherwise.
Let $A \in \mathfrak{X}(K)$ and $F : K \times K \to \mathbb{R}$ be a bifunction. We consider the problem of finding $x \in K$ such that
\begin{equation}\label{eq problem}
	x \in EP(F) \cap A^{-1}({\bf 0}),
\end{equation}
that is, $x$ is simultaneously an equilibrium point of $F$ and  a singularity of $A$.   In this paper we will assume that  $D(T^F_r)$ is closed geodesic convex, then  the set $\Omega$ is closed and geodesic convex by Theorem \ref{thm resolvent bifunction} and $A^{-1}({\bf 0})$ is closed and geodesic convex.

\begin{alg}\label{alg 1}
	Let a vector field $A \in \mathfrak{X}(K) $ and $F : K \times K \to \mathbb{R}$ be a bifunction. Choose an initial point  $x_0 \in K$ and define $\{x_n\}$, $\{y_n\}$ and $\{z_n\}$ as follows:
	\begin{equation}\label{eq alg 1}
		y_n : = \exp_{x_n}\alpha_{n}\exp^{-1}_{x_n}J^A_{\lambda_n}(x_n),
	\end{equation}
	\begin{equation} \label{eq alg 2}
	z_n \in K \  {\rm such \ that} \	F(z_n,t) - \frac{1}{r_n}\langle \exp^{-1}_{z_n}y_n, \exp^{-1}_{z_n}t \rangle \geq 0, \quad \forall t \in K,
	\end{equation}
	\begin{equation} \label{eq alg 3}
x_{n+1} := \exp_{x_n}\beta_n\exp^{-1}_{x_n}z_n, \quad \forall n \in \mathbb{N},
	\end{equation}
	where $\{\alpha_n\}, \{\beta_n\}, \{\lambda_n\}$ and $\{r_n\}$ are given real positive sequences such that
	\begin{itemize}
		\item[{\rm (i)}] $0 < a \leq \alpha_n, \beta_n \leq b <1, \quad \forall n \in \mathbb{N}$,
		\item [{\rm(ii)}] $0 < \hat{\lambda}  \leq \lambda_n \leq \tilde{\lambda} < \infty, \quad \forall n \in \mathbb{N}$,
		\item [{\rm(iii)}] $\liminf_{n \to \infty} r_n > 0.$
	\end{itemize}
\end{alg}

When $F \equiv 0$, the Algorithm \eqref{alg 1} becomes  the following algorithm for finding a solution of the problem \eqref{eq inclusion 2}.
\begin{alg}\label{alg 2}
	Let a vector field $A \in \mathfrak{X}(K)$.  Choose initial point  $x_0 \in K$ and define $\{x_n\}$ as follows:
	\begin{eqnarray}\label{eq alg 4}
	x_{n+1} &: =&  \exp_{x_n}\alpha_{n}\exp^{-1}_{x_n}J^A_{\lambda_n}(x_n), \quad \forall n \in \mathbb{N},
	\end{eqnarray}
	where $\{\alpha_n\} \subset (0,1)$ and $\{\lambda_n\} \subset (0, \infty)$ are the same as in Algorithm \ref{alg 1}.
\end{alg}

When $A \equiv {\bf 0}$,  the Algorithm \eqref{alg 1} becomes  the following algorithm for finding a solution of the problem \eqref{eq EP}.
\begin{alg}\label{alg 3}
	Let  $F : K \times K \to \mathbb{R}$ be a bifunction. Choose initial point  $x_0 \in K$ and define $\{x_n\}$ and $\{z_n\}$ as follows:
	\begin{eqnarray} \label{eq alg 5}
	&&z_n \in K \  {\rm such \ that} \ F(z_n,t) - \frac{1}{r_n}\langle \exp^{-1}_{z_n}x_n, \exp^{-1}_{z_n}t \rangle \geq 0, \quad \forall t \in K, \notag \\
	&&x_{n+1} :=  \exp_{x_n}\beta_n\exp^{-1}_{x_n}z_n, \quad \forall n \in \mathbb{N},
	\end{eqnarray}
	where $\{\beta_n\} \subset (0,1)$ and $\{r_n\} \subset (0, \infty)$ are the same as in Algorithm \ref{alg 1}.
\end{alg}

\begin{thm}\label{thm 1}
	Let a vector field $A \in \mathfrak{X}(K)$ be a maximal monotone. Let $F : K \times K \to \mathbb{R}$ be a bifunction satisfies (A1)--(A6) and $T^F_{r_n}$ be the resolvent of $F$ for $\{r_n\} \subset (0,\infty) $ with $D(T^F_{r_n})$ is closed and geodesic convex set  such that $\Omega \neq \emptyset.$ Then the sequence generated by Algorithm \ref{alg 1} converges to a solution of problem \eqref{eq problem}. 
\end{thm}
\begin{proof}
	It is  sufficient to show by Lemma \ref{lem fejer} that $\{x_n\}$ is Fej\'{e}r convergent with respect  to $\Omega$ and the cluster points of $\{x_n\}$ belongs to $\Omega$. We divide the proof into the following four steps.\\
	
	\framebox{{\bf Step I.}  We show that $\{x_n\}$ is Fej\'{e}r convergent with respect to $\Omega$.} \\
	Let $\omega \in \Omega$. Then $\omega \in EP(F)$ and $\omega \in A^{-1}({\bf 0})$. By Theorem \ref{thm resolvent bifunction} and Lemma \ref{lem resolvent bifunction}, we have $z_n = T^F_{r_n}(y_n)$ and
	\begin{eqnarray}\label{eq thm 1}
		d(z_n,\omega) &=& d(T^F_{r_n}(y_n), T^F_{r_n}(\omega)) \notag \\
		&\leq& d(y_n,\omega), \quad \text{for} \ \omega \in \Omega.
	\end{eqnarray}
	Since $\omega \in A^{-1}({\bf 0})$, Remark \ref{rem resolvent} gives $\omega = J^A_{\lambda_n}(\omega)$. Set $u_n := J^A_{\lambda_n}(x_n)$ and let $\bigtriangleup \left(\omega, x_n, u_n\right)$ $\subseteq M$  be  a geodesic triangle with vertices $\omega, x_n$ and $u_n$, and let $\bigtriangleup\left(\overline{\omega},\overline{x_n}, \overline{u_n} \right) \subseteq \mathbb{R}^2$ be a comparison triangle. Then, we have
	\begin{equation}\label{eq thm 2}
	d(x_n,\omega) = \| \overline{x_n} - \overline{\omega} \|, \quad d\left(x_n, u_n\right) = \left\|\overline{x_n} - \overline{u_n}\right\|, \quad d\left(u_n, \omega\right) = \left\|\overline{u_n} - \overline{\omega}\right\|.
	\end{equation}
	Recall from \eqref{eq alg 1} that $	y_n = \exp_{x_n}\alpha_{n}\exp^{-1}_{x_n}u_n$, then we have
	$$\overline{y_n} = (1-\alpha_n) \overline{x_n} + \alpha_n\overline{u_n}.$$
	From  \eqref{eq angle} and \eqref{eq less}, we get
	\begin{equation}\label{eq thm 3}
	\angle_\omega\left(u_n,x_n\right) \leq \angle_{\overline{\omega}}\left(\overline{u_n}, \overline{x_n}\right)
	\end{equation}
	and
	$$d(y_n,\omega) \leq \|\overline{y_{n}} - \overline{\omega}\|.$$ 
	From the last inequality, \eqref{eq thm 3} and $\alpha_n \in (0,1)$,  we have
	\begin{eqnarray}\label{eq thm 4}
	d^2(y_n,\omega) &\leq& \|\overline{y_{n}} - \overline{\omega}\|^2 \notag \\
	&=& \left\| (1-\alpha_n) \overline{x_n} + \alpha_n \overline{u_n} - \overline{\omega}\right\|^2 \notag \\
	&=&\left\| (\overline{x_n} - \overline{\omega})  - \alpha_n\left(\overline{x_n} -  \overline{u_n}\right) \right\|^2 \notag \\
	&=& \|\overline{x_n} - \overline{\omega}\|^2 + \alpha_n^2\left\|\overline{x_n} - \overline{u_n}  \right\|^2 -2 \alpha_n \|\overline{x_n} - \overline{\omega}\| \left\|\overline{x_n} - \overline{u_n} \right \| \cos\angle_{\overline{\omega}}\left(\overline{u_n}, \overline{x_n}\right)  \notag \\
	&\leq& \|\overline{x_n} - \overline{\omega}\|^2 + \alpha_n\left\|\overline{x_n} - \overline{u_n}  \right\|^2 -2 \alpha_n \|\overline{x_n} - \overline{\omega}\| \left\|\overline{x_n} - \overline{u_n} \right \| \cos\angle_{\overline{\omega}}\left(\overline{u_n}, \overline{x_n}\right)  \notag \\
	&=& \|\overline{x_n} - \overline{\omega}\|^2 + \alpha_n\left\|\overline{x_n} - \overline{u_n}  \right\|^2 -2 \alpha_n \left \langle \overline{x_n} - \overline{\omega}, \overline{x_n} - \overline{u_n} \right \rangle_{\mathbb{R}^2} \notag \\
	&=& \|\overline{x_n} - \overline{\omega}\|^2 + (\alpha_n - 2\alpha_n)\left\|\overline{x_n} - \overline{u_n}  \right\|^2 +2 \alpha_n \left \langle  \overline{\omega} - \overline{u_n} , \overline{x_n} - \overline{u_n} \right \rangle_{\mathbb{R}^2} \notag \\
	&\leq& d^2(x_n,\omega)  - \alpha_nd^2\left(x_n,u_n\right) +  2 \alpha_n\left \langle \exp^{-1}_{u_n} \omega , \exp^{-1}_{u_n} x_n \right \rangle. 
	\end{eqnarray}
	On the other hand, since $u_n := J^A_{\lambda_n}(x_n)$ and $J^A_{\lambda_n}$ is firmly nonexpansive, it follows from Proposition \ref{prop firmly = 0}  that
	$$\left \langle \exp^{-1}_{u_n} \omega , \exp^{-1}_{u_n} x_n \right \rangle \leq 0.$$
	This together with \eqref{eq thm 4} yield that
	\begin{eqnarray}
	d^2(y_n,\omega)  &\leq& d^2(x_n,\omega)  -\alpha_nd^2\left(x_n,u_n\right)  \label{eq thm 5} \\
	&\leq & d^2(x_n,\omega).  \label{eq thm 6}
	\end{eqnarray}
	Recall from \eqref{eq alg 1} that $y_n = \exp_{x_n}\alpha_{n}\exp^{-1}_{x_n}u_n$, then we have $d(x_n,y_n) = \alpha_nd\left(x_n,u_n\right)$. From \eqref{eq thm 5}, we get
	\begin{eqnarray}\label{eq thm 7}
	d^2(y_n,\omega)  &\leq& d^2(x_n,\omega) -  \frac{1}{\alpha_n}d^2\left(x_n,y_n\right). 
	\end{eqnarray}

	For $n \in \mathbb{N}$, let $\gamma_n : [0,1]  \to M$ be a geodesic joining $\gamma_n(0) = x_n$ to $\gamma_n(1) = z_n$. Then, \eqref{eq alg 3} can be written as $x_{n+1} = \gamma_n(\beta_n)$. By using geodesic convexity of Riemannian distance, \eqref{eq thm 1} and \eqref{eq thm 6}, we get
\begin{eqnarray}\label{eq thm 8}
	d(x_{n+1},\omega) &=& d(\gamma_n(\beta_n), \omega) \notag \\
	&\leq&(1-\beta_n) d(\gamma_n(0),\omega) + \beta_nd(\gamma_n(1),\omega) \notag \\
	&=& (1-\beta_n) d(x_n,\omega) + \beta_nd(z_n,\omega) \notag \\
	&\leq& (1-\beta_n) d(x_n,\omega) + \beta_nd(y_n,\omega) \notag \\
	&\leq&(1-\beta_n) d(x_n,\omega) + \beta_nd(x_n,\omega) \notag \\
	&=& d(x_n,\omega).
\end{eqnarray}
Therefore, $\{x_n\}$ is Fej\'{e}r convergent with respect  to $\Omega$.\\

\framebox{{\bf Step II.} We show that $\lim_{n \to \infty}d(x_{n+1},x_n) = 0.$} \\
Fix  $n \in \mathbb{N}$. Let  $\bigtriangleup(x_n,z_n,\omega)$ be a geodesic triangle with vertices $x_n$, $z_n$ and $\omega$, and $\bigtriangleup(\overline{x_n},\overline{z_n},\overline{\omega})$ be the corresponding comparison triangle. Then, we have 
$$d(x_n,\omega) = \|\overline{x_n} - \overline{\omega}\|, \quad d(z_n,\omega) = \| \overline{z_n} - \overline{\omega} \| \ {\rm and} \ d(z_n,x_n) = \|\overline{z_n} - \overline{x_n} \|.$$
Recall that $x_{n+1} := \exp_{x_n}\beta_n\exp^{-1}_{x_n}z_n$, so it comparison point is $\overline{x_{n+1}} = (1-\beta_n) \overline{x_n} + \beta_n\overline{z_n}$.
 Using \eqref{eq less}, \eqref{eq thm 1}, and \eqref{eq thm 6}, we get
\begin{eqnarray}
	d^2(x_{n+1},\omega) &\leq& \|\overline{x_{n+1}} - \overline{\omega}\|^2 \notag \\
	&=& \|(1-\beta_n)\overline{x_n} + \beta_n\overline{z_n} - \overline{\omega}\|^2 \notag \\
	&=& \|(1-\beta_n)(\overline{x_n} - \overline{\omega}) + \beta_n(\overline{z_n} - \overline{\omega})\|^2 \notag \\
	&=& (1-\beta_n) \|\overline{x_n} - \overline{\omega}\|^2 + \beta_n\|\overline{z_n} - \overline{\omega}\|^2 -\beta_n(1 -\beta_n)\|\overline{x_n} - \overline{z_n}\|^2 \notag \\
	&=& (1-\beta_n)d^2(x_n,\omega) + \beta_n d^2(z_n, \omega) -\beta_n(1 -\beta_n)d^2(x_n , z_n) \notag \\
	&\leq& (1-\beta_n) d^2(x_n,\omega) + \beta_n d^2(y_n, \omega) -\beta_n(1 -\beta_n)d^2(x_n, z_n) \label{eq thm 9} \\
	&\leq&(1-\beta_n)d^2(x_n ,\omega) + \beta_nd^2(x_n , \omega) -\beta_n (1-\beta_n)d^2(x_n , z_n) \notag \\
	&=& d^2(x_n ,\omega) -\beta_n(1 -\beta_n)d^2(x_n , z_n).  \label{eq thm 10}
\end{eqnarray}

From \eqref{eq thm 10}, we also obtain
$$\beta_n(1 -\beta_n)d^2(x_n , z_n) \leq d^2(x_n,\omega) - d^2(x_{n+1},\omega),$$
and we further have
\begin{eqnarray*}
	d^2(x_n , z_n) &=& \frac{1}{\beta_n(1 -\beta_n)}(d^2(x_n,\omega) - d^2(x_{n+1},\omega)) \\
	&\leq&  \frac{1}{a(1 -b)}(d^2(x_n,\omega) - d^2(x_{n+1},\omega)).\notag 
\end{eqnarray*}
Since $\{x_n\}$ is a Fej\'{e}r convergent with respect to $\Omega$ which implies that $\lim_{n \to \infty}d(x_n,\omega)$ exists.  By letting $n \to \infty$, we have
\begin{equation}\label{eq thm 11}
	\lim_{n \to \infty}d(x_n,z_n) = 0.
\end{equation}

Recall that $x_{n+1} = \gamma_n(\beta_n)$ for all $n \in \mathbb{N}$, using the  geodesic convexity of Riemannian distance, we obtain
\begin{eqnarray*}
	d(x_{n+1},x_n) &=& d(\gamma_n(\beta_n),x_n) \notag \\
	&\leq& (1-\beta_n) d(\gamma_n(0),x_n) + \beta_nd(\gamma_n(1),x_n) \notag \\
	&=&  (1-\beta_n)d(x_n,x_n) + \beta_n d(z_n,x_n) \notag \\
	&=&  \beta_n d(x_n,z_n) \notag \\
	&\leq& b d(x_n,z_n).
	\end{eqnarray*}
Letting  $n \to \infty$ and using \eqref{eq thm 11}, we get
\begin{equation}\label{eq thm 12}
	\lim_{n \to \infty} d(x_{n+1},x_n) = 0.
\end{equation}

\framebox{{\bf Step III.} We show that $\lim_{n \to \infty}d(x_n,y_n) = 0.$}\\
Using \eqref{eq thm 7} and \eqref{eq thm 9}, we obtain
\begin{eqnarray*}
	d^2(x_{n+1},\omega) &\leq& (1-\beta_n) d^2(x_n,\omega) + \beta_n(d^2(x_n,\omega) -  \frac{1}{\alpha_n}d^2\left(x_n,y_n\right)) \notag \\
	&& -\beta_n(1 -\beta_n)d^2(x_n, z_n) \notag \\
	&=& (1-\beta_n) d^2(x_n,\omega) + \beta_nd^2(x_n,\omega)  - \frac{\beta_n}{\alpha_n}d^2\left(x_n,y_n\right) -\beta_n(1 -\beta_n)d^2(x_n, z_n) \notag \\
	&=& d^2(x_n,\omega) - \frac{\beta_n}{\alpha_n}d^2\left(x_n,y_n\right) -\beta_n(1 -\beta_n)d^2(x_n, z_n). 
\end{eqnarray*}
With some rearrangements we  obtain
\begin{equation*}
\frac{a}{b}d^2(x_n,y_n) \leq \frac{\beta_n}{\alpha_n}d^2\left(x_n,y_n\right)  \leq  d^2(x_n,\omega) -  d^2(x_{n+1},\omega) -\beta_n(1 -\beta_n)d^2(x_n, z_n). 
\end{equation*}
The Fej\'{e}r convergent of $\{x_n\}$ with respect to $\Omega$ and \eqref{eq thm 11} together imply that
\begin{equation}\label{eq thm 13}
	\lim_{n \to \infty}d(x_n,y_n) = 0.
\end{equation}

\framebox{{\bf Step IV.}  We show that the cluster points of $\{x_n\}$ belongs to $\Omega$.}\\
Since the sequence $\{x_n\}$ is Fej\'{e}r convergent, by condition (ii) of Lemma \ref{lem fejer}, $\{x_n\}$ is bounded. Hence, there exists a subsequence $\{x_{n_i}\}$ of $\{x_n\}$ which converges to a cluster point $x^*$ of $\{x_n\}$. From \eqref{eq thm 13},  we get $y_{n_i} \to x^*$ as $i \to \infty$. Also from \eqref{eq thm 11}, implies that  $z_{n_i} \to x^*$ as $i \to \infty$.

We firstly prove that $x^* \in EP(F)$. By $z_n = T^F_{r_n}(y_n)$, we get 
$$F(z_n,y) - \frac{1}{r_n}\langle \exp^{-1}_{z_n}y_n, \exp^{-1}_{z_n}y \rangle \geq 0, \quad \forall y \in K.$$
Since a bifunction $F$ is monotone, we obtain that
$$- \frac{1}{r_n}\langle \exp^{-1}_{z_n}y_n, \exp^{-1}_{z_n}y \rangle \geq F(y,z_n).$$
Replacing $n$ by $n_i$, we get
\begin{equation}\label{eq thm 14}
	- \frac{1}{r_{n_i}}\langle \exp^{-1}_{z_{n_i}}y_{n_i}, \exp^{-1}_{z_{n_i}}y \rangle \geq F(y,z_{n_i}).
\end{equation}
Recall that
\begin{eqnarray*}
\lim_{i \to \infty}\|\exp^{-1}_{z_{n_i}}y_{n_i}\| = \lim_{i \to \infty} d(y_{n_i},z_{n_i}) = 0,
\end{eqnarray*}
so we get $\exp^{-1}_{z_{n_i}}y_{n_i} \to {\bf 0}$ as $i \to \infty$. Using $\liminf_{i \to \infty}r_{n_i} > 0$ and $y \mapsto F(x,y)$ is lower semicontinuous, and letting $i \to \infty$ into \eqref{eq thm 14}, we get
\begin{equation*}
	0 \geq \liminf_{i \to \infty}F(y,z_{n_i}) \geq F(y,x^*), \quad \forall y \in K.
\end{equation*}
Let $\gamma :[0,1] \to M$ be the  geodesic joining $\gamma(0) = x^*$ to $\gamma(1) = y \in K$. Since $K$ is geodesic convex, then $\gamma(t) \in K$ and $F(\gamma(t),x^*) \leq 0$ for all $t \in [0,1]$.
Since $y \mapsto F(x,y)$ is geodesic convex, we have, for $t > 0,$ the following
\begin{eqnarray*}
	0 = F(\gamma(t),\gamma(t)) &\leq& tF(\gamma(t),y) + (1-t)F(\gamma(t),x^*) \\
	&\leq& tF(\gamma(t),y).
\end{eqnarray*}
Dividing by $t$ and since $x \mapsto F(x,y)$ is upper semicontinuous, we see that
\begin{eqnarray*}
	0 &\leq&  \limsup_{t \to 0^+} F(\gamma(t),y)  \\
	&\leq& F(x^*,y).
\end{eqnarray*}
Since $y \in K$ is chosen arability,  $x^* \in EP(F)$.

Next, we prove that $x^* \in A^{-1}({\bf 0})$.  Since $\{\alpha_{n}\} \subset (0,1)$ satisfying $0 < a \leq \alpha_n \leq b <1$,  $\frac{1}{\alpha_{n}}d(x_n,y_n) = d(x_n,u_n)$, and $\lim_{n \to \infty}d(x_n,y_n) =0$, we may see that
\begin{equation}\label{eq thm 15}
	\lim_{n \to \infty}d(x_n,u_n) = 0.
\end{equation}
Since $\hat{\lambda} \leq \lambda_n \leq \tilde{\lambda}$, we may assume without the loss of generality that $\lim_{i \to \infty}\lambda_{n_i} = \lambda$ for some subsequence $\{\lambda_{n_i}\}$ of $\{\lambda_n\}$ and some $\lambda \in [\hat{\lambda},\tilde{\lambda}]$.  Recall that $u_n = J^A_{\lambda_n}(x_n)$. Then by \eqref{eq thm 15} and Lemma \ref{lem convergent resolvent}, we obtain
we obtain $\lim_{i \to \infty}{u_{n_i}} = x^*$ and that $ x^* = J^A_{\lambda}(x^*)$.
From Remark \ref{rem resolvent}, we obtain $x^* \in A^{-1}({\bf 0})$. Therefore, we get $x^* \in \Omega$. By a (iii) of Lemma \ref{lem fejer},  the sequence $\{x_n\}$ generated by Algorithm \ref{alg 1} converges to a solution of the problem \eqref{eq problem}. The proof is therefore completed.
\end{proof}

Next, we have the following results of Theorem \ref{thm 1} as follows:

\begin{cor}\label{cor 1}
	Let a vector field $A \in \mathfrak{X}(K)$ be a maximal monotone such that $A^{-1}({\bf 0}) \neq \emptyset.$ Then the sequence generated by Algorithm \ref{alg 2} converges to a solution of problem \eqref{eq inclusion 2}. 
\end{cor}

\begin{cor}\label{cor 2}
	  Let $F : K \times K \to \mathbb{R}$ be a bifunction satisfies (A1)--(A6) and $T^F_{r_n}$ be the resolvent of $F$ for $\{r_n\} \subset (0,\infty) $ with $D(T^F_{r_n})$ is closed and geodesic convex set such that $EP(F) \neq \emptyset.$ Then the sequence generated by Algorithm \ref{alg 3} converges to a solution of problem \eqref{eq EP}. 
\end{cor}

\section{Applications}\label{sec5}
In this section, we derive an algorithm for finding the minimizers of minimization problems, and also give an algorithm for finding the saddle points of minimax problems.
\subsection{Minimization problems. \\}
Let  $g : M \to \mathbb{R}$ be a proper,  lower semicontinuous geodesic convex function. Consider the optimization problem:
\begin{equation}\label{eq mini pb}
\underset{x \in M}{\min} \  g(x).
\end{equation}

We denote $S_g$ the solution set of  \eqref{eq mini pb}, that is,
\begin{equation}\label{eq mini solution}
S_g = \{x \in M : g(x) \leq g(y), \quad \forall y \in M\}. 
\end{equation}

\begin{dfn}\label{sub diff} Let $g : M \to \mathbb{R}$ be a geodesic convex and $x \in M$. A vector $s \in T_xM$ is called a {\it subgradient} of $g$ at $x$ if and only if
	\begin{equation}\label{eq sub diff}
	g(y) \geq g(x) + \langle s , \exp^{-1}_xy \rangle, \quad \forall y \in M.
	\end{equation}
\end{dfn}
\noindent The set of all subgradients of $g$, denoted by $\partial g(x)$ is called the {\it subdifferential} of $g$ at $x$, which is closed  geodesic convex (possibly empty) set.

\begin{lem}\label{lem sub diff} {\rm \cite{MR2506692}}
	Let $g:M \to \mathbb{R}$ be a proper, lower semicontinuous geodesic convex function. Then, the subdifferential $\partial g$ of $g$ is a maxiaml monotone vector field.
\end{lem}
It is easy to see that
$$x \in S_g \Longleftrightarrow {\bf 0} \in \partial g(x).$$

Recall that $\partial g$ is maximal monotone if $g: M \to \mathbb{R}$ is lower semi-continuous and convex. Applying Algorithm \ref{alg 1} to the multivalued vector filed $\partial g$, we obtain the following results for the convex minimization problem \eqref{eq mini pb}.

\begin{thm} Let $g : M \to \mathbb{R}$ be a proper, lower semicontinuous  geodesic convex function and  $F : K \times K \to \mathbb{R}$ be a bifunction satisfying (A1)--(A6) and $T^F_{r_n}$ be the resolvent of $F$ for $\{r_n\} \subset (0,\infty) $ with $D(T^F_{r_n})$ is closed and geodesic convex set such that $EP(F) \cap S_g \neq \emptyset.$ Let $\{x_n\}$ be a sequence in $D(g)$ generated as 
	$$	y_n : = \exp_{x_n}\alpha_n\exp^{-1}_{x_n}J^{\partial g}_{\lambda_n}(x_n),$$
	$$	z_n \in K \  {\rm such \ that} \	F(z_n,t) - \frac{1}{r_n}\langle \exp^{-1}_{z_n}y_n, \exp^{-1}_{z_n}t \rangle \geq 0, \quad \forall t \in K,$$
	$$	x_{n+1} := \exp_{x_n}\beta_n\exp^{-1}_{x_n}z_n, \quad \forall n \in \mathbb{N}, $$
where $\{\alpha_n\}, \{\beta_n\}, \{\lambda_n\}$ and $\{r_n\}$ are real positive sequences such that
\begin{itemize}
	\item[{\rm (i)}] $0 < a \leq \alpha_n, \beta_n \leq b <1, \quad \forall n \in \mathbb{N}$,
	\item [{\rm(ii)}] $0 < \hat{\lambda}  \leq \lambda_n \leq \tilde{\lambda} < \infty, \quad \forall n \in \mathbb{N}$,
	\item [{\rm(iii)}] $\liminf_{n \to \infty} r_n > 0.$
\end{itemize}
	Then, the sequence $\{x_n\}$ converges to a solution of the problem $EP(F) \cap S_g$.
\end{thm}

\begin{cor}
	Let $g : M \to \mathbb{R}$ be a proper, lower semicontinuous geodesic convex function and $ S_g \neq \emptyset.$ Let $\{x_n\}$ be a sequence in $D(g)$ generated as
	\begin{eqnarray*}
			x_{n+1} &:=& \exp_{x_n}\alpha_n\exp^{-1}_{x_n}J^{\partial g}_{\lambda_n}(x_n), \quad \forall n \in \mathbb{N}, 
	\end{eqnarray*} 
where $\{\alpha_n\}$ and $\{\lambda_n\}$ are real positive sequences such that
\begin{itemize}
	\item[{\rm (i)}] $0 < a \leq \alpha_n \leq b <1, \quad \forall n \in \mathbb{N}$,
	\item [{\rm(ii)}] $0 < \hat{\lambda}  \leq \lambda_n \leq \tilde{\lambda} < \infty, \quad \forall n \in \mathbb{N}$.
\end{itemize}
	Then, the sequence $\{x_n\}$ converges to a solution of the problem \eqref{eq mini pb}.
\end{cor}
\subsection{Saddle points in a minimax problem. \\}
In this subsection, we first recall the formulation of saddle point problems in the frame work of Hadamard manifolds. Then we derive on algorithm to find the saddle point.
Our results improves the results given by Li et al. \cite{MR2506692}.

Let $M_1$ and $M_2$ be the Hadamard manifolds, and $K_1$ and $K_2$ the geodesic convex subset of $M_1$ and $M_2$, respectively.  A function $H: K_1 \times K_2 \to \mathbb{R}$ is called a {\it saddle function} if 
\begin{itemize}
	\item[{\rm (a)}] $H(x,\cdot)$ is geodesic convex on $K_2$ for each $x \in K_1$ and
	\item[{\rm (b)}] $H(\cdot,y)$ is geodesic concave, i.e., $-H(\cdot,y)$ is geodesic convex on $K_1$ for each $y \in K_2.$
\end{itemize}
A point $ \tilde{z} = (\tilde{x},\tilde{y})$ is said to be a {\it saddle point} of $H$ if
$$H(x,\tilde{y}) \leq H(\tilde{x},\tilde{y}) \leq H(\tilde{x},y), \quad  \forall z = (x,y) \in K_1 \times K_2.$$
We denote $SPP$ to the set of saddle points of $H$. Let $V_H : K_1 \times K_2 \to 2^{TM_1} \times 2^{TM_2}$ be a multivalued vector field associated with saddle function $H$, defined by
\begin{equation}\label{eq saddle}
	V_H(x,y) = \partial(-H(\cdot,y))(x) \times \partial(H(x,\cdot))(y), \quad \forall (x,y) \in K_1 \times K_2.
\end{equation}
The product space $M = M_1 \times M_2$ is a Hadamard manifold and the tangent space of $M$ at $z = (x,y)$ is $T_zM = T_xM_1 \times T_yM_2$. For further details, see \cite[Page 239]{MR1390760}. The corresponding metric given by 
$$\langle w,w'\rangle = \langle u, u'\rangle + \langle v,v'\rangle, \quad \forall w = (u,v), w' = (u',v') \in T_zM.$$
A geodesic in the product manifold $M$ is the product of two geodesic in $M_1$ and $M_2.$  Then, for any two point $z = (x,y)$ and $z' = (x',y')$ in $M$, we have
$$\exp_z^{-1}z' = \exp_{(x,y)}^{-1}(x',y') = (\exp_x^{-1}x', \exp_y^{-1}y').$$
A vector field $V : M_1 \times M_2 \to 2^{TM_1} \times 2^{TM_2}$ is said to be monotone if and only if for any $z = (x,y), z'=(x',y'), w = (u,v) \in V(z)$ and $w' = (u',v') \in V(z'),$ we have  
$$\langle u,\exp_x^{-1}x' \rangle + \langle v,\exp_y^{-1}y' \rangle \leq \langle u',-\exp_{x'}^{-1}x \rangle + \langle v',-\exp_{y'}^{-1}y \rangle.$$
\begin{thm}{\rm \cite{MR2506692}}
	Let $H$ be a saddle function on $K = K_1 \times K_2$ and $V_H$ the multivalued vector field defined by \eqref{eq saddle}. Then $V_H$ is maximal monotone.
\end{thm}

One can check that a point $\tilde{z} = (\tilde{x},\tilde{y}) \in K$ is a  saddle point of $H$ if and only if it is a singularity of $V_H$.    Applying Algorithm \eqref{alg 1} to multivalued vector field $V_H$ associated with the saddle function $H$, we get the following result.

\begin{thm}
	Let $H : K = K_1 \times K_2 \to \mathbb{R}$ be a saddle function and $V_H: K_1 \times K_2 \to 2^{TM_1} \times 2^{TM_2}$ be the associated maximal monotone vector field. Let $F : K \times K \to \mathbb{R}$ be a bifunction satisfying (A1)--(A6) and $T^F_{r_n}$ be the resolvent of $F$ for $\{r_n\} \subset (0,\infty)$ with $D(T^F_{r_n})$ is closed and geodesic convex set such that $EP(F)\cap SSP \neq \emptyset.$ Choose initial point $x_0 \in K \times K$ and define $\{x_n\}, \{y_n\}$ and $\{z_n\}$ as follows: $$	y_n : = \exp_{x_n}\alpha_n\exp^{-1}_{x_n}J^{\partial V_H}_{\lambda_n}(x_n),$$
	$$	z_n \in K \  {\rm such \ that} \	F(z_n,t) - \frac{1}{r_n}\langle \exp^{-1}_{z_n}y_n, \exp^{-1}_{z_n}t \rangle \geq 0, \quad \forall t \in K,$$
	$$	x_{n+1} := \exp_{x_n}\beta_n\exp^{-1}_{x_n}z_n, \quad \forall n \in \mathbb{N}, $$
	where $\{\alpha_n\}, \{\beta_n\}, \{\lambda_n\}$ and $\{r_n\}$ are real positive sequences such that
	\begin{itemize}
		\item[{\rm (i)}] $0 < a \leq \alpha_n, \beta_n \leq b <1, \quad \forall n \in \mathbb{N}$,
		\item [{\rm(ii)}] $0 < \hat{\lambda}  \leq \lambda_n \leq \tilde{\lambda} < \infty, \quad \forall n \in \mathbb{N}$,
		\item [{\rm(iii)}] $\liminf_{n \to \infty} r_n > 0.$
	\end{itemize}
	Then, the sequence $\{x_n\}$ converges to a solution of the problem $EP(F) \cap SPP$.
\end{thm}

\begin{cor}
Let $H : K = K_1 \times K_2 \to \mathbb{R}$ be a saddle function and $V_H: K_1 \times K_2 \to 2^{TM_1} \times 2^{TM_2}$ be the associated maximal monotone vector field such that $SSP \neq \emptyset.$ Choose initial point $x_0 \in K$ and define $\{x_n\}$ as follows: 
$$	x_{n+1}  : = \exp_{x_n}\alpha_n\exp^{-1}_{x_n}J^{\partial V_H}_{\lambda_n}(x_n), \quad \forall n \in \mathbb{N},$$
where $\{\alpha_n\}$and $ \{\lambda_n\}$ are real positive sequences such that
\begin{itemize}
	\item[{\rm (i)}] $0 < a \leq \alpha_n \leq b <1, \quad \forall n \in \mathbb{N}$,
	\item [{\rm(ii)}] $0 < \hat{\lambda}  \leq \lambda_n \leq \tilde{\lambda} < \infty, \quad \forall n \in \mathbb{N}$.
\end{itemize}
Then, the sequence $\{x_n\}$ converges to a saddle point of $H$.
\end{cor}

\section*{Acknowledgments}
The first author was financially supported by the Research Professional Development Project Under the Science Achievement Scholarship of Thailand (SAST). 
This project was supported by Center of Excellence in Theoretical and Computational Science (TaCS-CoE), Faculty of Science, KMUTT.

 %%%%%%%%%%%%%%%%%%%%%%%%%%%%%%%%%%%%%%%%%%%%%%%%%%%%%%%%%%%%%%%%%%%%%%%%%%%%%%%%%%%%%%%%%%%%%%%%%%%%%%%%%%%%
%%\bibliographystyle{vancouver}
%%\bibliography{ref}

\end{document}